\documentclass[12 pt]{article}
\usepackage[margin = 2.35cm]{geometry}
\usepackage{thm-restate}
\usepackage{amsthm}
 \usepackage{dsfont}
 \usepackage[T1]{fontenc}
 \usepackage[utf8]{inputenc}
\usepackage{amsmath,amssymb,thmtools, latexsym,color,epsfig,enumerate, graphicx}

\usepackage{tikz}
\usetikzlibrary{calc}
\usepackage{comment}
\usepackage{enumerate}
\newcommand{\cupdot}{\mathbin{\mathaccent\cdot\cup}}

\usepackage[hidelinks]{hyperref}
\usepackage[square,sort,comma,numbers]{natbib}

\setcitestyle{square}

\newcommand{\eps}{\ensuremath{\varepsilon}}

\newtheorem{thm}{Theorem}[section]
\newtheorem{Definition}[thm]{Definition}
\newtheorem{lemma}[thm]{Lemma}

\newtheorem{claim}{Claim}[section]

\newtheorem{corollary}[thm]{Corollary}
\newtheorem{conjecture}[thm]{Conjecture}

\newtheorem{remark}[thm]{Remark}

\date{}
\title{Rolling backwards can move you forward: on embedding problems in sparse expanders }
\author{%
  Nemanja Dragani\'c \thanks{Department of Mathematics, ETH, 8092 Z\"urich, Switzerland. Email: \href{mailto:nemanja.draganic@math.ethz.ch} {\nolinkurl{nemanja.draganic@math.ethz.ch}}.}
\and Michael Krivelevich \thanks{School of Mathematical Sciences, Sackler Faculty of Exact Sciences, Tel Aviv University, Tel Aviv 6997801, Israel. Email: \href{mailto:krivelev@tauex.tau.ac.il}{\nolinkurl{krivelev@tauex.tau.ac.il}}. Supported in part by USA-Israel BSF grant 2018267, and by ISF grant 1261/17.}
\and 
Rajko Nenadov \thanks{Department of Mathematics, ETH, 8092 Z\"urich, Switzerland. Email: \href{mailto:rajko.nenadov@math.ethz.ch} {\nolinkurl{rajko.nenadov@math.ethz.ch}}.}
}

\begin{document}
\maketitle

\begin{abstract}\small\baselineskip=9pt
We develop a general embedding method based on the Friedman-Pippenger tree embedding technique (1987) and its algorithmic version, essentially due to Aggarwal et al. (1996), enhanced with a roll-back idea allowing to sequentially retrace previously performed embedding steps. We use this method to obtain the following results.
\begin{itemize}
    \item We show that the size-Ramsey number of logarithmically long subdivisions of bounded degree graphs is linear in their number of vertices, settling a conjecture of Pak (2002).
    \item  We give a deterministic, polynomial time online algorithm for finding vertex-disjoint paths of prescribed length between given pairs of vertices in an expander graph. Our result answers a question of Alon and Capalbo (2007).
    \item We show that relatively weak bounds on the spectral ratio of $d$-regular graphs force the existence of a topological minor of $K_t$ where $t=(1-o(1))d$. We also exhibit a construction which shows that the theoretical maximum $t=d+1$ cannot be attained even if $\lambda=O(\sqrt{d})$. This answers a question of Fountoulakis, K\"uhn and Osthus (2009).
\end{itemize}
\end{abstract}

\section{Introduction}

Given a graph $H$ from some class of graphs, and a graph $G$ with specific properties, is there a copy of $H$ in $G$? In other words, does there exist an embedding of $H$ into $G$? This general question is one of the central settings of combinatorics. Embedding questions lie at the heart of many classical problems, in particular problems in graph Ramsey theory and Tur\'an-type extremal theory. 

We will consider embedding problems where the host graph $G$ is sparse, i.e. the number of edges in $G$ is linear in its number of vertices. This is a natural and important setup both for theoretical and practical reasons, and its potential applicability ranges from problems in extremal combinatorics like Ramsey-type problems, to construction of lean but resilient networks in computer networking.

In particular, we will work with sparse expanders --- those are sparse graphs in which all sets of vertices $S$ of (up to) a certain size have a relatively large neighborhood. For a comprehensive source of information about expanders, see the survey of Hoory, Linial and Wigderson \cite{hoory2006expander}. Closely related to expander graphs is the notion of pseudo-random graphs. Informally, a graph is pseudo-random if it behaves similarly to a random graph when it comes to edge distribution. A very popular class of examples of such graphs are $(n,d,\lambda)$-graphs, introduced by Alon \cite{alon1986eigenvalues}. 
\begin{Definition}
A $d$-regular graph $G$ on $n$ vertices is an $(n,d,\lambda)$-\emph{graph} if all of the eigenvalues of its adjacency matrix, except the largest one, are at most $\lambda$ in absolute value. 
\end{Definition}
 One can show that the smaller $\lambda$ is, the closer the graph resembles a random graph in terms of edge distribution (see Section \ref{sec:subdivisions intro} for some details). A small $\lambda$ also means that the graph has good expansion properties and we will use a few such results throughout the paper. For a survey on pseudo-random graphs, see the paper of Krivelevich and Sudakov \cite{krivelevich2006pseudo}.

For our embedding problems, usually the host graph $G$ will be an $(n,d,\lambda)$-graph, for sufficiently small $\lambda$ and a constant $d$. What kind of subgraphs can we hope to find in such graphs? One natural restriction will be that the girth of the graph we embed is $\Omega(\log n)$, as there exist $(n,d,\lambda)$-graphs with small \emph{spectral ratio} $\lambda/d$ and of logarithmic girth, as shown in the seminal paper of Lubotzky, Phillips and Sarnak \cite{MR963118}. Thus we are normally confined to embedding trees and other graphs with large girth.

There is a large body of research devoted to finding (almost-spanning and spanning) bounded degree trees in sparse expanders and in sparse random graphs. Beck \cite{beck1983size} used results about long paths in expanding graphs to argue that one can find monochromatic linear sized paths in 2-colored sparse random graphs. Friedman and Pippenger \cite{friedman1987expanding} proved an analogous statement for arbitrary bounded degree trees in sparse expanders, which was improved upon by Haxell \cite{haxell2001}, who showed that under similar assumptions one can embed even larger trees into (sparse) expanders.
Alon, Krivelevich and Sudakov \cite{alon2007embedding} proved the existence of every almost spanning tree of bounded degree in both sparse random graphs and in appropriate $(n,d,\lambda)$-graphs, later improved by Balogh, Csaba and Samotij and Pei \cite{balogh2010large}, and for a resilience version of this result see \cite{balogh2011local}. Finally, for random graphs $G\sim G(n,p)$ with $p=\frac{C\log n}{n}$ and for a fixed $d$, Montgomery \cite{montgomery2019spanning} recently proved that for large enough $C$, $G$ typically contains all spanning trees of maximum degree at most $d$, resolving an old conjecture of Kahn (see \cite{kahn2016}). For results about finding small minors of logarithmic girth in sparse expanders, see, e.g. \cite{montgomery2015logarithmically,shapira2015small}.

In this paper, we present three different results related to embedding into sparse expanders --- the first one deals with size-Ramsey numbers of logarithmic subdivisions of bounded degree graphs and resolves a conjecture of Pak from 2002 \cite{pak2002}, while the second is concerned with the classical problem of finding vertex-disjoint paths in graphs, and resolves a problem of Alon and Capalbo from 2007 \cite{alon2007finding}. The third one is about finding topological minors of complete graphs in $(n,d,\lambda)$-graphs, and is related to a question of Fountoulakis, K\"uhn and Osthus \cite{FKO09}. For those three problems we develop two variations of our embedding technique. Both are based on the result of Friedman and Pippenger \cite{friedman1987expanding} about embedding trees in expander graphs vertex by vertex and an idea by Daniel Johannsen \cite{Daniel}, which allows us to successively remove vertices from the list of already embedded vertices. This \emph{roll-back} result turns out to be very powerful for tackling problems of this sort. One of the variants which we show is algorithmic, and uses ideas by Dellamonica and Kohayakawa \cite{dellamonica2008algorithmic}, who showed an algorithmic version of the original Friedman-Pippenger embedding result, by reducing it to a certain online matching problem solved by Aggarwal et al. \cite{aggarwal1996efficient}.

\subsection{Size-Ramsey numbers of subdivided graphs}

Given a graph $H$ and an integer $k \ge 2$, a graph $G$ is said to be \emph{$k$-Ramsey} for $H$ if every coloring of the edges of $G$ with $k$ colors contains a monochromatic copy of $H$. This notion was introduced by Ramsey~\cite{ramsey1930problem}, who proved that for every graph $H$ there exists $N \in \mathbb{N}$ such that $K_N$ is $k$-Ramsey for $H$. The smallest such $N$, denoted by $R_k(H)$, is called the \emph{Ramsey number}. Determining the asymptotic order of $R_2(K_\ell)$ is one of the most important open problems in this area~\cite{conlon2009new,spencer1975ramsey}. We will be concerned with the related notion of \emph{size-Ramsey numbers}, introduced by Erd\H{o}s, Faudree, Rousseau and Schelp \cite{erdos1978size}. Given a graph $H$  and an integer $k\geq 2$, the \emph{size-Ramsey number} $\hat R_k(H)$ is the smallest integer $m$ such that there exists a graph $G$ with $m$ edges which is $k$-Ramsey for $H$. The existence of the Ramsey number immediately implies the upper bound $\hat R_k(H) \le \binom{R_k(H)}{2}$. Other related notions include Folkman numbers, chromatic-Ramsey numbers, degree-Ramsey numbers, etc. We refer the reader to a survey by Conlon, Fox and Sudakov \cite{conlon2015recent} for a thorough treatment of the topic.

Answering a \$$100$ question of Erd\H{o}s \cite{erdHos1981combinatorial}, Beck~\cite{beck1983size} showed that paths have linear size-Ramsey number, that is $\hat R_2(P_n) \le C n$ for an absolute constant $C$. He also raised the question~\cite{beck1990size} of whether $\hat R_2(H)$ grows linearly for graphs with bounded maximum degree. This was proven for trees by Friedman and Pippenger \cite{friedman1987expanding} and for cycles by Haxell, Kohayakawa and \L uczak \cite{haxell1995induced}. However, the general case was settled in the negative by R\"odl and Szemer\'edi \cite{rodl2000size}, who showed that there exists a constant $c > 0$ such that for every sufficiently large $n$ there is a graph $H$ with $n$ vertices and maximum degree $3$ for which $\hat {R}_2(H) \ge n \log^c n$. In the same paper, they conjectured that $\log^c n$ can be improved to $n^{\varepsilon}$ for some constant $\varepsilon > 0$, but this remains open.
\subsubsection{Subdivisions of graphs}\label{sec:subdivisions intro}
Since we are far from understanding size-Ramsey numbers of bounded degree graphs in general, one natural step in this direction is to consider subdivisions of those graphs. Given a graph $H$ and a function $\sigma \colon E(H) \rightarrow \mathbb{N}$, the \emph{$\sigma$-subdivision} $H^{\sigma}$ of $H$ is the graph obtained from $H$ by replacing each edge $e \in E(H)$ with a path of length $\sigma(e)$ joining the endpoints of $e$, such that all these paths are mutually vertex-disjoint (except possibly at the endpoints). In other words, each edge $e\in E(H)$ is \emph{subdivided} $\sigma(e) - 1$ times.

Size-Ramsey numbers of `short' subdivisions were first studied by Kohayakawa, Retter and R\"odl~\cite{kohayakawa2016size}. In a recent paper \cite{draganic2020size} we improved their bounds by showing that $\hat{R}_k(H^q)\leq O(n^{1+1/q})$, for constant $q,k$ and for all bounded degree graphs $H$, thus removing a polylogarithmic factor from their bound and answering their question. In general, these graphs were considered in the context of Ramsey theory by Burr and Erd\H{o}s~\cite{erdosburr} and by Alon \cite{alon1994subdivided}.

In Section \ref{sec:proof1} (Theorem \ref{thm:pak}) we show that bounded degree graphs with $n$ vertices such that every two vertices of degree $\ge 3$ are at distance $q=\Omega(\log n)$ have linear size-Ramsey numbers (in their order). In fact we prove a stronger result on arbitrary long \emph{subdivisions} of bounded degree graphs, answering a conjecture of Pak \cite{pak2002} along the way. He conjectured that long subdivisions of bounded degree graphs have linear size-Ramsey number.

\begin{conjecture}[\cite{pak2002}] \label{conj:pak}
	For every $k, D \in \mathbb{N}$ there exist $C, L > 0$ such that if $H$ is a graph with $\Delta(H) \le D$ and  $\sigma(e) = \ell \ge L \log (v(H^\sigma))$ for all $e \in E(H)$ then $\hat R_k(H^\sigma) \le C v(H^\sigma)$.
\end{conjecture}

Pak \cite{pak2002} showed $\hat R_k(H^\sigma) = O(v(H^\sigma) \log^3 (v(H^\sigma)))$, and the special case where $H$ is a fixed (small) graph and $\sigma(e)$ grows was resolved by Donadelli, Haxell and Kohayakawa \cite{donadelli2005note}. 

We show that every \emph{$\eta$-uniform} graph on $n$ vertices is $k$-Ramsey for $H^\sigma$ with $v(H^\sigma) \le \alpha n$ and $\sigma(e) \ge \log n$, for some small $\alpha > 0$. As a typical random graph with $n$ vertices and $m = Cn$ edges is $\eta$-uniform, for sufficiently large $C$, there are an abundance of $\eta$-uniform graphs with $O(n)$ edges, thus confirming Conjecture \ref{conj:pak}.

\begin{Definition}
Given $0 < \eta \le 1$, we say that a graph $G$ with $n$ vertices and density $p = e(G) / \binom{n}{2}$ is \emph{$\eta$-uniform} if for every pair of disjoint subsets $U, W \subseteq V(G)$ of size $|U|,|W| \ge \eta n$, we have
$$
	\big|e(U, W) - |U||W|p\big| \le \eta |U||W|p.
$$
\end{Definition}
Now we state our result.
\begin{restatable}{thm1}{Pak}\label{thm:pak}
	For every $k,D \in \mathbb{N}$ and for every $\delta>0$, there exist $\eta, \alpha, C > 0$, such that the following holds for every $\eta$-uniform graph $G$ with $n$ vertices and $m \ge C n$ edges: every $k$-edge-coloring of $G$ contains a monochromatic copy of every graph $H^\sigma$, where $H$ is a graph with maximum degree at most $D$, $v(H^\sigma) \le \alpha n$ and $\sigma(e) \ge \delta \log n$ for every $e \in E(H)$.
\end{restatable}

 Besides random graphs, explicit constructions of $\eta$-uniform graphs of constant average degree are also known. One class of examples of such graphs are $(n,d,\lambda)$-graphs, for suitably chosen parameters. Indeed, the well known Expander Mixing Lemma \cite{alon1988explicit} states that for every $(n,d,\lambda)$-graph $G$ and for every $U,V\subseteq V(G)$ it holds:
\begin{equation}\label{EML}
 \left|e_G(U,V)-{\frac {d|U||V|}{n}}\right|\leq \lambda {\sqrt {|U||V|}}.
\end{equation}
From this, one can see that every $(n,d,\lambda)$-graph is $\eta$-uniform for $\eta=2\sqrt{\frac{\lambda}{d}}$. Hence, for a fixed $d$, the parameter $\lambda$ is accountable for the uniformity of the distribution of the edges of a $d$-regular graph. But how small can $\lambda$ be in terms of $d$, so that there exists a $(n,d,\lambda)$-graph? One can show that $\lambda=\Omega(\sqrt{d})$ for every such graph whenever $d<0.99n$, and there are known constructions of $d$-regular graphs for which $\lambda$ attains this bound, and $n$ is arbitrarily large. This provides us examples of bounded degree graphs, which are $\eta$-uniform (for $\eta\sim d^{-1/4}$). For several constructions of such graphs, see, e.g., \cite{krivelevich2006pseudo}. As discussed above, there are known constructions of such graphs which have logarithmic girth, showing that our result is asymptotically tight with respect to the bound on $\sigma$. Indeed, if $H$ is a triangle, and $G$ is a graph with girth strictly larger than $c\log n$, then $G$ does not contain $H^\sigma$ for any $\sigma$ bounded from above uniformly by $c\log n/3$. 

Note that Theorem \ref{thm:pak} is in fact a \emph{universality} result, meaning that the $\eta$-uniform graph in question is $k$-Ramsey for \emph{all} graphs in the class we are interested in, hence our theorem confirms Pak's conjecture in a strong way. Furthermore, from our proof it can be seen that we actually find a monochromatic subgraph of the graph we color, which contains all described subdivided graphs. Extending the definition in \cite{kohayakawa2011sparse}, we say that a graph $G$ is \emph{$k$-partition universal} for a class of graphs $\mathcal{F}$ if for every $k$-coloring of the edges of $G$, there exists a monochromatic subgraph of $G$ which contains a copy of every graph in $\mathcal{F}$. Under this framework, we actually prove that the graph we color is up to a constant factor the optimal $k$-partition universal graph for the class of all described subdivisions of graphs. For further universality-type results in Ramsey theory see for example \cite{conlon1size, draganic2020size,kohayakawa2011sparse,kohayakawa2016size}.

\subsection{The vertex-disjoint paths problem}
For a given graph $G$ and a collection of $k$ disjoint pairs of vertices $(a_i,b_i)$ from $G$, can we find for each $i$ a path from $a_i$ to $b_i$, such that the found paths are all vertex-disjoint?
This decision problem is $\mathcal{NP}$-complete \cite{garey1979computers} when $G$ is allowed to be an arbitrary graph. Furthermore, it remains $\mathcal{NP}$-complete, even when $G$ is restricted to be in the class of planar graphs. For fixed $k$, it is shown to be in $\mathcal{P}$ \cite{seymour1995graph}. 
A variant of this problem in random graphs was studied independently by Hochbaum \cite{hochbaum1992exact}, and by Shamir and Upfal \cite{shamir1985fast}. Both papers proved that for a fixed set of at most $O(\sqrt{n})$ disjoint pairs of vertices in the random graph $G(n,m)$, with high probability (whp) there exist vertex-disjoint paths between every pair if $m>Cn\log n$, for a constant $C>1$. Subsequently, Broder, Frieze, Suen and Upfal \cite{broder1996efficient} improved this result:

\begin{thm}[\cite{broder1996efficient}]\label{Frieze}
There exist $\alpha,\beta>0$, such that whp the following holds. Let $G=G(n,m)$ for $m=(\log n+\omega(1))\frac{n}{2}$, and let $d=2m/n$. For every collection $\mathcal{F}$ of at most $\alpha \frac{n\log d}{\log n}$ disjoint pairs of vertices $(a_i,b_i)$ in $G$, there exists a path for every $i$ connecting $a_i$ to $b_i$, such that all paths are vertex-disjoint, if the following condition is satisfied for every vertex $v\in V(G)$: 
\begin{align*}
&|N_G(v)\cap (A\cup B)|<\beta d_G(v)\qquad \text{where }\quad A=\cup_i\{a_i\}\text{ and } B=\cup_i \{b_i\}.
\end{align*}
\end{thm}
This is an improvement over the mentioned previous papers in many aspects. The number of pairs is optimal up to a constant factor --- most pairs in $G(n,m)$ are at distance $\Omega(\log n/\log d)$, so in general one can hope to connect at most $O(\frac{n\log d}{\log n})$ pairs $(a_i,b_i)$ in a graph on $n$ vertices. Furthermore, the pairs $(a_i,b_i)$ are not fixed before generating $G(n,m)$, but are rather chosen adversarily after having exposed a random graph $G\sim G(n,m)$. The last constraint is also optimal up to a constant factor --- if the adversary chooses $a_1$ and $b_1$ to be at distance $2$, and then chooses all neighbours of $a_1$ to be in other pairs from $\mathcal{F}$, then obviously one cannot find the requested disjoint paths. The bound on the number of edges $m$ is also asymptotically optimal, as this many edges are needed for $G$ to be connected whp.

Changing the focus to the sparse(r) setting, Alon and Capalbo \cite{alon2007finding} studied graphs with constant average degree with good expansion properties. In particular, they proved that for any graph $G$ which is a $d$-blowup\footnote{The $d$-blowup of a graph $G$ is the graph obtained from $G$ by replacing each vertex $v$ with an independent set $I_v$ of size $d$, and replacing every edge $(v,u)$ with a complete bipartite graph between $I_v$ and $I_u$.} of a $(n,d,\lambda)$-graph with a small spectral ratio, and any collection of $O\left(\frac{nd\log d}{\log n}\right)$ pairs of vertices in $G$ which satisfy a similar local condition like in \cite{broder1996efficient}, one can connect those pairs with vertex-disjoint paths. The number of pairs is optimal up to a constant factor, and they provide an online polynomial time algorithm for finding them.

The argument of Alon and Capalbo does not allow to control the length of the paths found by the algorithm. Accordingly, they ask for a similar result where the length of the paths between each pair is at most $O(\log n)$. In Section \ref{sec:VDP}, we prove such a result (Theorem \ref{Thm:VDP in (N,D,lambda)-graphs}). Furthermore, we do it not only for blowups, but directly for $(n,d,\lambda)$-graphs for $\lambda<d^{0.99}$. We get the optimal dependency on $n$ and $d$, both for the number of pairs and for the upper bound on the length of the paths. Finally, our algorithm is online in the sense that an adversary can choose the pairs one by one, where the next pair is given after the connection for the previous one is established.

\begin{restatable}{thm1}{VDP}\label{Thm:VDP in (N,D,lambda)-graphs}
Let $0<\varepsilon<\frac{1}{4 }$, and let $G$ be an $(n,d,\lambda)$-graph, with $\lambda<d^{1-\varepsilon}/320$ and $d^\varepsilon\geq5$. Let $S$ be any set of vertices which satisfies $|N_G(x)\cap S|\leq \frac{d}{4}$ for every $x\in V(G)$. Let $P=\{a_i,b_i\}$ be a collection of at most $\frac{\varepsilon n\log d}{480\log n}$ disjoint pairs of vertices from $S$. There exists a polynomial time algorithm to find vertex-disjoint paths in $G$ between every pair of vertices $\{a_i,b_i\}$, such that the paths are of equal length which is less than $5\frac{\log n}{\varepsilon\log d}$. Furthermore, the pairs $(a_i,b_i)$ are given one by one, and the next pair is revealed when the previous connection is made; all established connections cannot be changed.
\end{restatable}

These results are closely related to the study of non-blocking networks, which arise in a variety of applications, including construction of communication networks and distributed-memory architectures. For some results see, e.g., \cite{pippenger1982telephone,feldman1986non,feldman1988wide}. In contrast to our results, the graphs which are usually considered here have pre-determined sets of vertices ("inputs" and "outputs") from which the pairs are chosen, while the pairs in our result can be chosen by an adversary, but in such a way that they satisfy an essentially minimal local property. Besides that, the path lengths in some constructions of non-blocking networks are also of optimal $O(\log n)$ size \cite{arora1990line}. Hence, in some sense our results are a common generalization of \cite{arora1990line} and \cite{alon2007finding},
as we both allow the adversary to choose the pairs, and our paths are logarithmic in size, although our algorithm is less efficient than the one in \cite{arora1990line}. 
It can also be seen from our proof that we can allow the adversary to terminate arbitrary already established connections between pairs (and thus freeing the used vertices in the corresponding paths) a finite number of times during the mentioned online algorithm. This feature is related to the study of permutation networks \cite{aggarwal1996efficient}.

A lot of attention has also been paid to the edge-disjoint paths problem. For a short survey, see \cite{frieze1998disjoint}, and for a result on edge-disjoint paths in sufficiently strong expander graphs see \cite{alon2007finding}.

\subsection{Topological minors}
We will use the tools developed for proving Theorem \ref{Thm:VDP in (N,D,lambda)-graphs} to derive results about topological minors in expander graphs. We say that a graph $H$ is a \emph{topological minor} of a graph $G$ if there is a subdivision of $H$ which is a subgraph in $G$; the vertices in this subdivision which correspond to vertices of $H$ we call \emph{branching vertices}. Given an $(n,d,\lambda)$-graph $G$, how large can $t$ be such that $K_t$ is a topological minor of $G$? Since $G$ is $d$-regular, the best one could hope for is $t=d+1$.
We show that there is a topological minor of $K_t$ in $(n,d,\lambda)$-graphs for $t=(1-o(1))d$, assuming $\lambda=o(d)$. Furthermore, if $\lambda<d^{1-\varepsilon}$ we find a topological minor on the asymptotically smallest possible number of vertices $O(d^2\frac{\log n}{\log d})$, where all the paths which connect the branching vertices are of the same length $\ell$ with $\ell=O(\frac{\log n}{\log d})$. Note that we also need a non-trivial upper bound on $d$, since by a well known argument of Erd\H{o}s and Fajtlowitz \cite{ErdosF}, one can show that $d$ can be at most of order $\sqrt{n}$, if we want to guarantee a topological minor of $K_t$ with $t=(1-o(1))d$; for our argument we take $d=O(n^{1/5})$.

\begin{restatable}{thm1}{minors}\label{thm:minors in (N,D,lambda)}
Let $G$ be a $(n,d,\lambda)$-graph with $240\lambda<d\leq n^{1/5}/2$, and let $d_0$ be such that $d_0\geq 3$.
Then $G$ contains a topological minor of $K_t$ for $t=\lfloor d-80\lambda\sqrt{d_0}\rfloor$, the paths between branching vertices being of equal length $\ell$, where $\ell=O\left(\frac{\log n}{\log d_0}\right)$. 
\end{restatable}

Theoretically, one might hope to find topological minors of $K_{d+1}$, but we show that there exist $(n,d,\lambda)$-graphs with $\lambda=O(\sqrt{d})$ which do not contain a topological minor of $K_{d+1}$. This is related to a question of Fountoulakis, K\"uhn and Osthus \cite{FKO09}, for which we will need the following definition. We say that a graph $G$ is \emph{$(s, K)$-expanding} if for every subset $X \subseteq V(G)$ of size $|X| \le s$ we have $|N_G(X)| \ge K |X|$. They ask for determining the parameters $\alpha,K$ such that every $(\alpha n,K)$-expanding $d$-regular graph on $n$ vertices contains a topological minor of $K_{d+1}$. 
In particular, they ask whether constant expansion not depending on $d$ already forces a topological minor of $K_{d+1}$ in a $d$-regular graph.
We answer this question in the negative and show that there exist $d$-regular graphs with very strong expansion properties, but without a topological minor of $K_{d+1}$. In particular, for $\alpha=\alpha(d)>0$, we show the existence of $(2d+1)$-regular $(\alpha n,d-3)$-expanding graphs without a topological minor of $K_{2d+2}$.


\subsection{Outline of the paper and notation}

In Section \ref{sec:FP} we show two versions of our main embedding technique --- in Section \ref{sec:FP original} we show the non-algorithmic version of it, and in Section \ref{sec:algorithmic FP} we give an algorithmic version of the technique. In Section 3, we prove our main results --- Theorem \ref{thm:pak} (the resolution of Pak's conjecture) in Section \ref{sec:proof1}, and Theorem \ref{Thm:VDP in (N,D,lambda)-graphs} (vertex-disjoint paths in $(n,d,\lambda)$-graphs) in Section \ref{sec:VDP}. Furthermore, in Section \ref{sec:top minors} we present results about topological minors in expander graphs and in particular prove Theorem \ref{thm:minors in (N,D,lambda)}.

\paragraph{Notation.} We follow standard graph theoretic notation. In particular, given a graph $G$ and a vertex $x \in V(G)$, we denote by $N_G(x)$ the neighborhood of $x$ in $G$. For a subset of vertices $X \subseteq V(G)$ we denote by $\Gamma_G(X)$ the neighborhood of $X$, that is $\Gamma_G(X) = \bigcup_{x \in X} N_G(x) $, and we denote by $N_G(X)$ the external neighborhood of $X$, that is $N_G(X) = \Gamma_G(X) \setminus X$. By $\partial_G(x)$ we denote the set of edges incident with vertex $x$ in $G$. Given disjoint subsets of vertices $A, B \subseteq V(G)$, we denote by $e_G(A, B)$ the number of edges with one endpoint in $A$ and the other in $B$, and with $d_G(A, B) = e_G(A, B) / |A||B|$ the density of such a induced bipartite graph. We denote by $v(G)$ the number of vertices of $G$, and by $e(G)$ the number of edges of $G$.
Given graphs $G$ and $H$, we say that a mapping $\phi \colon V(H) \rightarrow V(G)$ is an \emph{embedding}, with the notation $\phi \colon H \hookrightarrow G$, if it is injective and preserves edges of $H$ (i.e. if $\{v,w\} \in E(H)$ then $\{\phi(v), \phi(w)\} \in E(G)$). 
For an embedding $\phi\colon H\hookrightarrow G$ and subsets $S_1\subseteq V(H), S_2\subseteq V(G)$ we denote by $\phi(S_1)$ the image of $S_1$ under $\phi$, and by $\phi^{-1}(S_2)$ the preimage of $S_2$ under $\phi$, i.e. $\phi(S_1)=\{y\in V(G)\mid \exists x\in S_1 \colon \phi(x)=y\}$, and $\phi^{-1}(S_2)=\{x\in V(H)\mid \phi(x) \in S_2\}$.
We omit floors and ceilings whenever it is not crucial.
Given two constant $\eps$ and $\alpha$, we use somewhat informal notation $\eps \ll \alpha$ to denote that $\eps$ is sufficiently small compared to $\alpha$. We denote by $\log n$ the natural logarithm of $n$ and by $\mathbb{N}$ the set of positive integers.

\section{Friedman-Pippenger type embedding theorems}\label{sec:FP}
Now we describe the main embedding machinery behind our proofs. It relies on the idea of Friedman and Pippenger, used for embedding trees in expanders vertex by vertex, by maintaining a certain invariant. An algorithmic version of this technique was presented by Dellamonica and Kohayakawa, based on a result about an online matching game by Aggarwal et al. \cite{aggarwal1996efficient}.
In the following two subsections, we give two Friedman-Pippenger type embedding theorems, non-algorithmic and algorithmic, enhanced with a roll-back idea, which allows us to sequentially retrace previously performed embedding steps. While the algorithmic result requires the host graph to have stronger expansion properties, it also enables us to embed larger graphs than with the technique described in Section \ref{sec:FP original}.

\subsection{The original Friedman-Pippenger theorem with rollbacks} \label{sec:FP original}
We start with a standard definition of expansion.
\begin{Definition}\label{def:expanding}
	Let $s \in \mathbb{N}$ and $K > 0$. We say that a graph $G$ is \emph{$(s, K)$-expanding} if for every subset $X \subseteq V(G)$ of size $|X| \le s$ we have $|N_G(X)| \ge K |X|$.
\end{Definition}
In order to develop our machinery, we define the notion of an \emph{$(s,D)$-good} embedding.
\begin{Definition}\label{deF:goodness}
Let $G$ be a graph and let $s, D \in \mathbb{N}$. Given a graph $F$ with maximum degree at most $D$, we say that an embedding $\phi \colon F \hookrightarrow G$ is \emph{$(s,D)$-good} if
\begin{equation} \label{eq:extendable}
	|\Gamma_G(X) \setminus \phi(F)| \ge \sum_{v \in X} \left[ D - \deg_F(\phi^{-1}(v)) \right]+|\phi(F)\cap X|
\end{equation}
for every $X \subseteq V(G)$ of size $|X| \le s$. Here we slightly abuse the notation by setting $\deg_F(\emptyset) := 0$, i.e. if a vertex $v\in V(G)$ is not used by $\phi$ to embed $F$, then we set $deg_F(\phi^{-1}(v))=0$.
\end{Definition}

We remark that the notion of a good embedding is the same as the one used by Friedman and Pippenger~\cite{friedman1987expanding} up to the last term on the right side of the inequality. The proof of the following theorem is almost identical to the one in \cite{friedman1987expanding}, and we present it in the Appendix for completeness.

\begin{thm} \label{thm:FP}
	Let $F$ be a graph with $\Delta(F) \le D$ and $v(F) < s$, for some $D, s \in \mathbb{N}$. Suppose we are given a $(2s-2,D{+}2)$-expanding graph $G$ and a $(2s-2,D)$-good embedding $\phi \colon F \hookrightarrow G$. Then for every graph $F'$ with $v(F') \le s$ and $\Delta(F') \le D$ which can be obtained from $F$ by successively adding a new vertex of degree $1$, there exists a $(2s-2,D)$-good embedding $\phi' \colon F' \hookrightarrow G$ which extends $\phi$.
\end{thm}

The second result we need is a simple corollary of the definition of $(s, D)$-goodness. While easy to prove, this observation \cite{Daniel} turns out to yield a powerful method for connecting vertices in expanding graphs. It has also been utilized in the recent paper by Montgomery \cite{montgomery2019spanning}
for embedding spanning trees in random graphs.

\begin{lemma} \label{lemma:delete}
	Suppose we are given graphs $G$ and $F$ with $\Delta(F)\leq D$, and an $(s, D)$-good embedding $\phi \colon F \hookrightarrow G$, for some $s, D \in \mathbb{N}$. Then for every graph $F'$ obtained from $F$ by successively removing a vertex of degree $1$, the restriction $\phi'$ of $\phi$ to $F'$ is also $(s, D)$-good.
\end{lemma}
\begin{proof}
	We show that the statement holds for the case where $F'$ is obtained from $F$ by removing a single vertex $v \in V(F)$ of degree $1$. The lemma then follows by iterating it. Let $\phi'$ be a restriction of $F$ to such $F'$, and let $w \in F'$ denote the unique neighbor of $v$. Let $X\subseteq V(G)$ with $|X|\leq s$.
	
	Assume first that $\phi(v)\notin X$. If $\phi(w)\notin X$ then both the left hand side (LHS) and the right hand side (RHS) of \eqref{eq:extendable} do not change. Otherwise (if $\phi(w) \in X$) the RHS of \eqref{eq:extendable} increases by $1$ (as the degree of $w$ in $F'$ is one less than it was in $F$). However, as $\phi(v)$ is no longer occupied (i.e. $\phi(v) \notin \phi'(F')$) and $\phi(v) \in N_G(\phi(w))$, the LHS also increases by one, hence the inequality again holds.
	
	Now, let $\phi(v)\in X$. If $\phi(w)\in X$, then the LHS increases by one, the first term on the RHS increases by two, and the last term decreases by one. Finally, if $\phi(w)\notin X$, the LHS grows by one, while the first term on the RHS also grows by one, and the last term drops by one. Hence, in every case the inequality holds.
 \end{proof}

\subsection{Algorithmic Friedman-Pippenger with roll-backs}\label{sec:algorithmic FP}
In this section we prove an algorithmic version of the embedding technique provided by Theorem \ref{thm:FP} and Lemma \ref{lemma:delete} from Section \ref{sec:FP original}. We start with a description of an online matching game, to which we reduce our embedding problem.

Let $m\geq 0$ be an integer. The game is played on a bipartite graph $H=(U\cup V, E)$. In the beginning we set $M$ (\emph{the current matching}) to be empty. At each step an adversary chooses a vertex $x\in U$ which is not covered by $M$, and we match it to some free vertex in $V$ to extend $M$. After each step the adversary is allowed to remove any number of edges from the current matching $M$, but at most $m$ times in total during the game. In \cite[Lemma 2.2.7]{aggarwal1996efficient}, Aggarwal et al. describe a polynomial time algorithm which finds a matching of size $n$, against any adversary, if $H$ satisfies the property that for each $X\subset U$ of size $|X|\leq n$, even if we remove at most half of the edges incident to every vertex in $X$, there are still at least $2|X|$ neighbors of $X$ in the obtained graph.

\begin{thm}[\cite{aggarwal1996efficient}, Aggarwal et al.]\label{matchingGame}
Let $H=(U\cup V,E)$ be a bipartite graph and let $n,m\in \mathbb{N}$, such that for every $X\subseteq U$ of size $|X|\leq n$ and for every $F\subseteq E$ such that $|F\cap \partial_H(x)|\leq d_H(x)/2$ for every $x\in X$, we have that $|N_{H-F}(X)|\geq 2|X|$. Then there is an algorithm which finds a matching of size $n$ against any adversary, if the adversary is allowed to remove edges from the matching at most $m$ times in total during the game. Furthermore, the number of operations which the algorithm performs is polynomial in $m+|V(H)|$.
\end{thm}

\begin{Definition}
We say that a graph $G=(V,E)$ has \emph{property $P_\alpha(n,d)$} if for every $X\subseteq V$ of size $|X|\leq n$ and every $F\subseteq E$ such that $|F\cap \partial_{G}(x)|\leq \alpha\cdot d_G(x)$ for every $x\in X$, we have $|N_{G-F}(X)|\geq 2d|X|$.
\end{Definition}

\begin{Definition}
Given a graph $G$, a subset of vertices $S\subseteq V(G)$, and natural numbers $n,m,d\in \mathbb{N}$, we define the following online game, which we call \emph{the $(G,S,n,m,d)$-forest building game}. At each step there is a forest $T\subseteq G$ (initially $T:=(S,\emptyset)$) with less than $n$ edges in $G$, and the adversary \emph{requests} a vertex $v\in T$ such that $d_T(v)<d$ and we are supposed to find a neighbor of $v$ in $V(G)-V(T)$, hence extending $T$ by a new leaf. The adversary is allowed to successively remove any number of vertices of degree 1 in $T$ after every step, but he is allowed to do so at most $m$ times in total, and none of the removed vertices are allowed to be in $S$. We win if at some point $T$ has $n$ edges.
\end{Definition}

The next theorem gives a handy tool for embedding forests algorithmically in a robust way. In comparison to the technique presented in Section \ref{sec:FP original}, here we require a stronger notion of expansion (the $P_\alpha(n,d)$-property) for the host graph, but the graphs we are embedding can have more vertices than before. The idea of the proof is similar to the one in \cite{dellamonica2008algorithmic}. 
\begin{thm}\label{AlgorithmicFP}
Let $\alpha,\beta>0$ with $\beta<2\alpha-1$ and let $G$ be a graph with property $P_\alpha(n,d)$. Let $S$ be a non-empty subset of vertices $S\subseteq V(G)$, such that for every vertex $x\in V(G)$ it holds that $|N_G(x)\cap S|\leq \beta\cdot d_G(x)$. Then there is an algorithm which wins the $(G,S,dn,m,d)$-forest building game after performing a number of operations polynomial in $m+|V(G)|$.
\end{thm}
\begin{proof}
In order to use Theorem \ref{matchingGame}, we construct the following auxiliary graph. Let $H$ be a bipartite graph with classes $U=V(G)\times [d]$ and $V=\{\Bar{v}\mid v\in V(G)- S\}$. In other words, $U$ consists of $d$ copies of $V(G)$, and $V$ is a copy of $V(G)- S$. Two vertices $(u,j)\in U$ and $\Bar{v}\in V$ are adjacent iff $\{u,v\}$ is an edge in $G$. Now we show that $H$ satisfies the condition of Theorem \ref{matchingGame} (with $dn$ instead of $n$).

Let $X\subseteq U$ be of size $|X|\leq dn$, and $F\subseteq E(H)$ be such that $|F\cap \partial_H(x)|\leq d_H(x)/2$ for every $x\in X$. We want to show that $|N_{H-F}(X)|\geq 2|X|$. 
By the pigeonhole principle, one of the $d$ copies of $V(G)$ in $U$ contains at least $|X|/d$ elements from $X$, or in other words, there is an $i\in [d]$ such that the set $X_i:=\{(u,i)\mid (u,i)\in X\}$ is of size $|X_i|\geq |X|/d$. Let $Y$ be an arbitrary subset of $X_i$ of size exactly $\lceil |X|/d\rceil$,
and let $Y'=\{u\mid (u,i)\in Y\}\subseteq V(G)$. 

We also define $F'\subseteq E(G)$ as follows:
\begin{align*}
F'=\Big\{\{u,v\}&\in E(G)\mid u\in Y',\,v\notin Y'  \text{, and } \{(u,i),\Bar{v}\}\in F \Big\}.
\end{align*}

Let $G'$ be the graph obtained from $G$ by deleting all edges in $F'$ and by deleting all edges which have one vertex in $Y'$ and the other in $S\setminus Y'$. 
Note the following facts:
\begin{itemize}
    \item[(\emph{i})] $|N_{H-F}(Y)|\geq |N_{G'}(Y')|$,
    \item[(\emph{ii})] $d_{G'}(x)\ge(1-\alpha)d_G(x)$ for all $x\in Y'$.
\end{itemize}
    The first claim is true as for every vertex $v\in N_{G'}(Y')$ there is a vertex $(u,i)\in Y$ such that $\{(u,i), \Bar{v}\}$ is an edge in $H-F$.
    For the second claim, notice that every vertex $x\in Y'$ after deleting the edges incident to $S\setminus Y'$ from $G$, loses at most $\beta\cdot d_G(x)$ edges, and after deleting $F'$ from $G$ it loses at most half of its remaining edges,
    which gives $d_{G'}(x)\geq \frac{1-\beta}{2}d_G(x)>(1-\alpha)d_G(x)$ edges.
    
    It follows from the second claim and from $|Y'|=\lceil|X|/d\rceil\leq \lceil nd/d\rceil=n$ (and from the assumption that $G$ has the $P_\alpha(n,d)$-property), that $|N_{G'}(Y')|\geq 2d|Y'|\geq 2|X|$. Together with (\emph{i}) this implies $|N_{H-F}(X)|\geq |N_{H-F}(Y)|\geq 2|X|$.

Now we reduce our forest building game on the graph $G$ to the matching game on the graph $H$. At the beginning our initial forest is set to be the empty graph on $S$, i.e. $T:=(S,\emptyset)\subseteq G$. We also set our auxiliary matching $M$ in $H$ to be empty in the beginning. During the game $M$ and $T$ will have the same number of edges. In each step the adversary requests a vertex $u\in T$ such that $d_T(u)<d$, and we want to find a vertex $v$ in $N_{G}(u)\setminus V(T)$ which extends $T$, such that $\{u,v\}$ is a new leaf in $T$. In order to do this, we find a vertex $(u,j)$ in $H$ for some $j\leq d$, which is not covered by $M$ (in the next paragraph we show that such a vertex exists), and we extend $M$ by finding a match $\Bar{v}\in V$ for $(u,j)$, using the algorithm from Theorem \ref{matchingGame}. Now we add the edge $(u,v)$ (which is in $G$ by the definition of $H$) to $T$. Note also that $v$ was not in $T$ before, as $v$ certainly is not in $S$ (by the definition of $\Bar{v}$), and for every other vertex $x\in T$, the vertex $\Bar{x}$ is covered by $M$, as $x$ has been added to $T$ by the same procedure, so $\Bar{x}\neq \Bar{v}$.

When the adversary wants to delete an edge $(u,v)$ (where $v$ is of degree $1$ in $T$) from $T$, then we also delete the corresponding edge $\{(u,j),\Bar{v}\}$ from $M$. Note that if at any step the adversary requests a vertex $u$ such that $d_T(u)<d$, then a vertex of the form $(u,i)$ (for some $i\in [d]$) has been used only at most $d-1$ times by the current matching $M$, so it is valid to assume that in each step we can find such a vertex which is not covered by $M$. Since the algorithm finds a forest $T$ with $dn$ edges at the same point when $M$ contains $dn$ edges, and we remove edges from the matching only at most $m$ times in total, thanks to Theorem \ref{matchingGame}, we are done.
\end{proof}

\section{Applications}
\subsection{Size-Ramsey number of long subdivisions}\label{sec:proof1}
Before we start with the proof of Theorem \ref{thm:pak}, we state a few preliminary results 
which will help us find a subgraph with good expansion properties in the edge colored graph in question.
\subsubsection{Preliminaries}

The following lemma tells us that if in a graph all sets of a specified size expand well, we can delete relatively few vertices, so that in the remaining graph all smaller sets also expand well. For related results see for example \cite{krivelevich2019expanders}. A similar statement also appeared in \cite{montgomery2014sharp}.

\begin{lemma} \label{lemma:clean_up}
	Let $G$ be a graph such that $|N_G(X)| \ge 3K s$	for every subset $X \subseteq V(G)$ of size $|X| = s$, for some $s\in\mathbb{N}$ and $K\geq 1$. Then there exists a subset $B \subseteq V(G)$ of size $|B| < s$ such that $G - B$ is $(s, K)$-expanding.	
\end{lemma}
\begin{proof}
Let $B\subset V(G)$ be a largest set such that $|N_G(B)|<K|B|$ and $|B|<s$ (or $B=\emptyset$ if no such set exists). We show that $H=G- B$ is $(s,K)$ expanding. Let $X\subseteq V(H)$ be an arbitrary non-empty set of size $|X|\leq s$ and suppose $|N_H(X)|<K|X|$. Then $|N_G(X\cup B)|<K|X|+K|B|=K|X\cup B|$, so by assumption we have $|X\cup B|\geq s$; let $S\cupdot R$ be a partition of $X\cup B$ with $|S|=s$. Therefore, we conclude:
\begin{align*}
    |N_H(X)|&\geq |N_G(X\cup B)|-|N_G(B)|\geq |N_G(S)|-|R|-Ks\geq
    3Ks-s-Ks\geq Ks\geq  K|X|
\end{align*}
which contradicts the assumption that $|N_H(X)|<K|X|$, so we are done.
\end{proof}

\noindent \textbf{Regular pairs}

The proof of Theorem \ref{thm:pak} combines results from Section \ref{sec:FP original} with a sparse version of Szemer\'edi's regularity lemma for multicolored graphs (or rather its corollary given shortly). 
\begin{Definition}
Given a graph $G$ and disjoint subsets $U, W \subseteq V(G)$, we say that the pair $(U, W)$ is \emph{$(G, \eps,p)$}-\emph{regular} for some $\eps, p \in (0, 1)$ if
$$
	|d_G(U', W') - d_G(U, W)| \le \eps p
$$
for every $U' \subseteq U$ of size $|U'| \ge \eps |U|$, and $W' \subseteq W$ of size $|W'| \ge \eps |W|$.
\end{Definition}
\begin{remark}\label{rem:huge-expansion}
If $U'\subseteq U$ and $W'\subseteq W$ are as above and $d_G(U,W) > \varepsilon p$, then there exists at least one edge between $U'$ and $W'$ in $G$, as otherwise $d_G(U', W')=0$, which contradicts $|d_G(U', W') - d_G(U, W)| \le \eps p$. It follows that $|N_G(U')|>(1-\varepsilon)|W|$.
\end{remark}

 The following corollary of Szemer\'edi's regularity lemma was proven in \cite[Lemma 3.4]{haxell1995induced}.

\begin{lemma} \label{lemma:reg}
	For every $k \ge 2$ and $0 < \eps < 1$, there exist $\mu, \eta > 0$ such that the following holds: Suppose $G = (V, E)$ is an $\eta$-uniform graph with $n$ vertices and density $p = e(G) / \binom{n}{2}>0$, and let $E = E_1 \cupdot E_2 \cupdot \ldots \cupdot E_k$ be an $k$-edge-coloring of $G$. Then, for some $1 \le z \le k$, there exist pairwise disjoint subsets $V_1, V_2, V_3 \subseteq V$ of size $|V_i| = \mu n$ such that 
 	\begin{enumerate}[(a)]
 		\item $(V_i, V_j)$ is $(G_z, \eps, p)$-regular, where $G_z = (V, E_z)$, and 
 		\item $d_{G_z}(V_i, V_j) \ge p|V_i||V_j|/2k$, 
 	\end{enumerate}
 	for every $1 \le i < j \le 3$. 	
\end{lemma}

We are ready to prove Theorem \ref{thm:pak}, which we restate here.
\Pak*
\subsubsection{Proof of Theorem \ref{thm:pak} --- resolution of Pak's conjecture}
\begin{proof}[Proof of Theorem \ref{thm:pak}]

Let $\mu = \mu(k, \eps)$ and $\eta = \eta(k, \eps)> 0$ be given by Lemma \ref{lemma:reg} for a sufficiently small constant $\eps \ll D^{-1}, k^{-1}$. Also assume w.l.o.g. that $D\gg 1/\delta$. Suppose we are given an $\eta$-uniform graph $G$ with $n$ vertices and a $k$-edge-coloring $E(G) = E_1 \cupdot E_2 \cupdot \ldots \cupdot E_k$, and let $1 \le z \le k$ and $V_1, V_2, V_3 \subseteq V(G)$ be obtained by applying Lemma \ref{lemma:reg}. In the rest of the proof we show that $\Gamma = (V(G), E_z)$ contains $H^\sigma$ for every $H$ satisfying conditions of the theorem with $\alpha = \eps \mu$. 

\paragraph{Prepare $\Gamma$.} Let $t = |V_i| = \mu n$. 
Let $\Gamma' = \Gamma[V_1, V_2]$ be a bipartite subgraph of $\Gamma$ induced by $V_1$ and $V_2$. From $(\Gamma, \eps, p)$-regularity of $(V_1, V_2)$ and from the assumption $\eps \ll 1/k, 1/D$, we conclude (Remark \ref{rem:huge-expansion}) that for every subset $X \subseteq V(\Gamma')$ of size $|X| = 2s$, where 
$$
	s = 2D^2 \eps t,
$$
we have 
$$
	|N_{\Gamma'}(X)| \ge t - \eps t - |X| \ge t/2 \ge 3(D + 3) |X|.
$$
Therefore, by Lemma \ref{lemma:clean_up} there exists a subset $B \subseteq V(\Gamma')$ of size $|B|=s$ such that $\Gamma_B = \Gamma' \setminus B$ is $(2s, D + 3)$-expanding. Let $V_1'=V_1\setminus B$ and $V_2'=V_2\setminus B$, so that $\Gamma_B=\Gamma_B[V_1',V_2']$. Most of $H^\sigma$ will be embedded using $\Gamma_B$ and the machinery from Section \ref{sec:FP original}, with occasional help from set $V_3$.

\paragraph{Embed $H$.} Consider a graph $H$ with maximum degree $D$ and let $\sigma \colon E(H) \rightarrow \mathbb{N}$ be a function such that $v(H^\sigma) = v(H) + \sum_{e \in E(H)} \sigma(e) < \eps t$ and $\sigma(e) \ge \delta \log n$ for every $e \in E(H)$. Let $(e_1, \ldots, e_m)$ be an arbitrary ordering of the edges of $H$, and for each $0 \le i \le m$ set $H_i = (V(H), \{e_1, \ldots, e_i\})$. Note that $H_0$ is just an empty graph on the vertex set $V(H)$. We inductively show that for each $0 \le i \le m$ there exists an embedding $\phi_i \colon H_i^\sigma \hookrightarrow \Gamma$ such that the following holds: 
\begin{enumerate}[(1)]
	\item $\phi_i(V(H)) \subseteq V_1'$, and \label{Pak:p1}
	\item the restriction of $\phi_i$ to $F_i = \phi_i^{-1}(V(\Gamma_B))$, denoted by $f_i \colon F_i \hookrightarrow \Gamma_B$, is $(2s-2,D)$-good. \label{Pak:p2}	
\end{enumerate}

Let us first prove the base case $i = 0$. Note that $H_0^\sigma = H_0$ is an empty graph on the vertex set $V(H)$. Let $a$ be a vertex (some new auxiliary vertex not used before) and $v \in V_1$, and set $\phi_0'(a) = v$. As $\Gamma_B$ is $(2s, D + 3)$-expanding, it is easy to see that $\phi_0'$ is a $(2s, D + 2)$-good embedding of a graph consisting of a single vertex. Let us extend such a one-vertex graph to a path $P$ of length $2\eps t$. By Theorem \ref{thm:FP}, there exists an $(2s, D + 2)$-good embedding $\phi_0' \colon P \hookrightarrow \Gamma_B$. Consider an arbitrary bijection between $V(H)$ and the set of \emph{odd} vertices in $P$ (i.e. the first vertex, third vertex, etc.). As $\phi_0'(a)$ is mapped into $V_1'$, all these vertices are also necessarily mapped into $V_1'$. Together with $\phi_0'$, such a bijection gives an embedding $\phi_0 \colon H_0 \hookrightarrow \Gamma_B$ with $\phi_0(V(H)) \subseteq V_1'$. As $\phi_0'$ was a $(2s, D + 2)$-good embedding, it is easy to verify that $\phi_0$ is a $(2s, D)$-good embedding, hence also a $(2s-2, D)$-good embedding.

Suppose the induction holds for some $i < m$ and let $e_{i+1} = \{a,b\}$. In short, we need to find a path from $\phi_i(a)$ to $\phi_i(b)$ of length $\sigma(e_{i+1})$, such that the part of it that goes through $\Gamma_B$ maintains $(2s-2,D)$-goodness. In the proof we  use  auxiliary parameters $\ell_1,\ell_2, h \in \mathbb{N}$, defined as follows: choose $h \in \mathbb{N}$ to be the smallest integer such that $(D-1)^h \ge \eps t$, and set $\ell_1 = \lfloor \sigma(e_{i+1}) / 2 \rfloor - h - 1$ and $\ell_2=\lceil \sigma(e_{i+1}) / 2 \rceil - h - 1$. Note that $\ell_1,\ell_2 > 1$ since $\lfloor \sigma(e_{i+1})/2\rfloor\geq \lfloor \delta \log n/2\rfloor\geq \log_{D-1}n>h+2$, where we used $1/\varepsilon\gg D\gg 1/\delta$.

Let $F_i = \phi_{i}^{-1}(\Gamma_B)$ be the part of $H_i^\sigma$ embedded into $\Gamma_B$, and $f_i$ be the restriction of $\phi_i$ to $F_i$. We start by constructing the graph $F_i'$ in two steps: First attach to $F_i$ two paths of lengths $\ell_1$ and $\ell_2$, one rooted in $a$ and the other in $b$, and let $a'$ and $b'$ denote the  other ends of such paths. Then attach two complete $(D-1)$-ary trees of depth $h$, one rooted in $a'$ and the other in $b'$. Let us denote the set of leaves of these trees by $L_a$ and $L_b$, respectively, and note that $|L_a| = |L_b| = (D-1)^h  \ge \eps t$ by the choice of $h$. Such trees have less than $(D-1)^{h + 1}  \le (D-1)^2\eps t$ vertices each, which together with a trivial bound $\ell_1 \leq \ell_2< \sigma(e_{i+1}) < v(H^\sigma)$ implies
\begin{align*}
	v(F_i') &\le v(F_i) + 2(\ell_2 - 1) + 2 \cdot ((D-1)^2\eps t-1)\\ &\le v(H^\sigma) + 2v(H^\sigma) + 2(D-1)^2\eps t < s.
\end{align*}
Assuming $D \ge 3$ each vertex has degree at most $D$ in $F_i'$ and, by its definition, $F_i'$ can be constructed from $F_i$ by successively adding a vertex of degree $1$. Therefore, we can apply Theorem \ref{thm:FP} to obtain a $(2s-2,D)$-good embedding $f_i' \colon F_i' \hookrightarrow \Gamma_B$ which extends $f_i$. 

Every vertex in $L_a$ is at distance exactly $\ell_1 + h$ from $a$, and every vertex in $L_b$ is at distance exactly $\ell_2 + h$ from $b$. Thus $f_i'(L_a) \subseteq V_{j_1}'$ and $ f_i'(L_b) \subseteq V_{j_2}'$ for some $j_1,j_2 \in \{1,2\}$. Next, we find a path of length $2$ from $f_i'(L_a)$ to $f_i'(L_b)$ with the internal vertex lying in $V_3$.  From $(\Gamma, \eps, p)$-regularity of the pairs $(V_1, V_3)$ and $(V_2,V_3)$, and $|f_i'(L_a)|, |f_i'(L_b)| \ge \eps t$, we know that all but at most $2\eps t$ vertices in $V_3 \setminus \phi_i(H_i^\sigma)$ are adjacent to both $f_i'(L_a)$ and $f_i'(L_b)$. As $|V_3| = t$ and $v(H^\sigma) < \eps t$, this implies that there exists a free vertex in $V_3$ adjacent both to $f_i'(L_a)$ and $f_i'(L_b)$, which gives a desired path of length $2$.

To summarize, we have found a path $P(x,y)$ of length $2$ from $f_i'(x)$ to $f_i'(y)$, for some $x \in L_a$ and $y \in L_b$, with the internal vertex avoiding $V_1 \cup V_2$ and $\phi_i(H_i^\sigma)$. By Lemma \ref{lemma:delete}, the restriction of $f_i'$ to the graph obtained by removing all newly added vertices to $F_i$ which do not lie either on the path from $x$ to $a$ or from $y$ to $b$ is $(2s-2,D)$-good. Together with the path $P(x,y)$, this defines an embedding $\phi_{i+1}$ of $H_{i+1}^\sigma$ into $\Gamma$.
\end{proof}

\subsection{Vertex-disjoint paths in expanding graphs}\label{sec:VDP}
Theorem \ref{AlgorithmicFP} provides a framework for embedding forests (in polynomial time) into graphs with certain expansion properties, while allowing arbitrary leaf deletions along the way. We present an application of this result to the classical problem of finding vertex-disjoint paths between given pairs of vertices in graphs.

Now we state the key result of this subsection. Theorem \ref{Thm:VDP in (N,D,lambda)-graphs} (stated in the introduction) will then follow directly from the properties of $(n,d,\lambda)$-graphs. 
\begin{thm}\label{Thm:VDP}
Let $G$ be a graph with the $P_\alpha(n,d)$ property for $3\leq d<n$, and such that for every two disjoint $U,V\subseteq V(G)$ of sizes $|U|,|V|\geq n(d-1)/16$ there exists an edge between $U$ and $V$. Let $S$ be any set of vertices such that $|N_G(x)\cap S|\leq \beta d_G(x)$ for every $x\in V(G)$ and
let $P=\{a_i,b_i\}$ be a collection of at most $\frac{dn\log d}{15\log n}$ disjoint pairs from $S$. If $\beta<2\alpha-1$ then there exists a polynomial time algorithm to find vertex-disjoint paths in $G$ between every pair of vertices $\{a_i,b_i\}$, such that the length of each path is $2\left\lceil\frac{\log (n/16)}{\log (d-1)}\right\rceil+3$. Furthermore, the pairs $(a_i,b_i)$ are given one by one, and the next pair is revealed when the previous connection is made; all established connections cannot be changed.
\end{thm}
\begin{proof}
By Theorem \ref{AlgorithmicFP}, there is an algorithm which works in time polynomial in $V(G)$, and wins the $(G,S,nd,n^2d^3,d)$-forest building game. We construct the required disjoint paths one by one as follows. 
Let $h$ be the smallest integer such that $(d-1)^h\geq\frac{n}{16}$.

For the first pair $\{a_1,b_1\}$ we find two disjoint complete $(d-1)$-ary trees of depth $h$ in $G$, rooted at $a_1$ and $b_1$, using the algorithm for winning the forest building game. The sets of leaves of those trees are then of size at least $n/16$ each. To each leave of each tree we then attach at least one but at most $d-1$ new edges, so that the resulting trees are of depth $h+1$, and have exactly $\frac{n(d-1)}{16}$ leaves. Therefore, by assumption, there is an edge connecting these sets of leaves, thus creating a path (between $a_1$ and $b_1$) of length $2h+3$. Remove from our current forest all other edges which do not lie on this path. We continue in the same fashion, by finding two complete $(d-1)$-ary trees rooted at $a_2$ and $b_2$ (disjoint from the path connecting $a_1$ and $b_1$), then finding a connecting edge between the sets of leaves, and removing all edges from the $(d-1)$-ary trees, which do not lie on the found path. We delete the edges successively, by always removing the edges which are incident with vertices of degree $1$, just like in the forest building game.

We do this procedure for every pair of vertices, and note that we can do this as at any given point the current forest which we use for our argument has at most
\[
   \frac{dn\log d}{15\log n}\cdot  (2h+3) + 2\cdot 4\cdot\frac{(d-1)n}{16}< \frac{dn}{2}+\frac{dn}{2}<dn
\]
edges, where the first term is a bound on the total number of edges used in previous paths, and the second one bounds the number of edges in the current $(d-1)$-ary trees we use. Furthermore we delete vertices of degree 1 at most $2|P|\cdot dn<n^2d^3$ times. This completes the proof.
\end{proof}

The following result can be derived from the Expander Mixing Lemma through rather routine calculations.
\begin{lemma}[\cite{dellamonica2008algorithmic}, Lemma 2.7]\label{lem:property}
Let $G$ be an $(n,d,\lambda)$-graph and let $d_0,n_0$ be positive integers. $G$ has property $P_\alpha(n_0,d_0)$ for $\alpha>0$ if the following holds:
\[
        1-\alpha>\frac{n_0(1+4d_0)}{2n}+\frac{\lambda}{d}(1+\sqrt{2d_0}).
\]
\end{lemma}

We are ready to give the promised proof of Theorem \ref{Thm:VDP in (N,D,lambda)-graphs}, restated below for the convenience of the reader.
\VDP*
\begin{proof}
Let $n_0=n/16d_0$ and $d_0=d^{\varepsilon}$. From Lemma \ref{lem:property} we see that $G$ has the $P_{3/4}(n_0,d_0)$-property. Furthermore, by the Expander Mixing Lemma (eq. (\ref{EML})), we have that for sets $U,V\subseteq V(G)$ of size at least $\frac{(d_0-1)n_0}{16}$ it holds:
\[
e_G(U,V)\geq \frac{d|U||V|}{n}- \lambda \sqrt{|U||V|}\geq \sqrt{|U||V|}\left(\frac{d(d_0-1)n_0}{16n}-\lambda\right)
\geq \frac{(d_0-1)d}{256d_0}-\lambda\]
which, together with $\lambda<d/320$ and $d_0\geq 5$ gives $e_G(U,V)>0$. Applying Theorem \ref{Thm:VDP} to $G$ completes the proof; here are the final calculations.
\begin{itemize}
    \item 
Number of pairs:
\begin{align*}
\frac{d_0n_0\log d_0}{15\log n_0}&=\frac{n\log d^{\varepsilon/2} }{16\cdot15\log(n/16d_0)}=\frac{\varepsilon n\log d }{480\log(n/16d_0)}>\frac{\varepsilon n\log d }{480\log n};
\end{align*}
\item 
Length of paths:
\begin{align*}
2\left\lceil\frac{\log (n_0/16)}{\log (d_0-1)}\right\rceil+3
&=2\left\lceil\frac{\log (n/256d_0)}{\log (d^{\varepsilon/2}-1)}\right\rceil+3\leq 5\frac{\log n}{\varepsilon\log d}. 
\end{align*}
\end{itemize}
\end{proof}

\subsection{Topological minors in $(n,d,\lambda)$-graphs}\label{sec:top minors}

Now we use Theorem \ref{Thm:VDP} from the previous subsection to prove Theorem \ref{thm:minors in (N,D,lambda)}, restated here for the reader's convenience.
\minors*
\begin{proof}
Let $T=\{v_1,\ldots,v_t\}$ be a set of $t$ vertices with pairwise distances at least four in $G$; one can find such a set by putting a vertex from $G$ into $T$ and removing all vertices at distance at most four from $G$, and repeating this process. For each vertex $v_i\in T$, let $S_i=\{v_i^j\mid 1\leq j\leq t, i\neq j\}$ be an arbitrary set of neighbors of $v_i$ of size $|S_i|=t-1$. Now consider the following set of pairs
\[
P=\{(v_i^j,v_{j}^i)\mid 1\leq i <j\leq t\},
\]
and note that for every two distinct $S_i$ and $S_j$ there is exactly one pair from $S_i\times S_j$ in $P$. We want to use Theorem \ref{Thm:VDP} to link the pairs in $P$, which evidently will give a topological minor of $K_t$ in $G$.
Note that $G$ has the $P_\alpha(n_0,d_0)$ property for $\alpha=1-\frac{40\lambda}{d}\sqrt{d_0}$ and $n_0=\frac{25\lambda n}{d_0d}$ by Lemma \ref{lem:property}; indeed, note that

\[
\frac{n_0(1+4d_0)}{2n}+\frac{\lambda}{d}(1+\sqrt{2\cdot d_0})<60\frac{\lambda}{d}+2\frac{\lambda}{d}\sqrt{d_0}<40\frac{\lambda}{d}\sqrt{d_0}=1-\alpha.
\]

Hence, since for every $x\in V(G)$ we have $|N_G(x)\cap (\bigcup P\cup T)|<\beta d$, for $\beta=1-\frac{80\lambda\sqrt{d_0}+1}{d}$, and $\beta<2\alpha-1$ holds, and since every two disjoint sets in $G$ of size $n_0(d_0-1)/16$ have an edge between them by (\ref{EML}), we indeed can apply Theorem \ref{Thm:VDP} and get a topological minor of $K_t$
such that the constructed paths between vertices in each pair in $P$ are of length $2\left\lceil\frac{\log (n_0/16)}{\log (d_0-1)}\right\rceil+3=O\left(\frac{\log n}{\log d_0}\right)$, completing the proof.
\end{proof}

By assuming that $\lambda<D/240$ and setting $d_0=3$ in the theorem above, one can get a topological minor of $K_t$ for $t=D-O(\lambda)$. Can we do better? If $\lambda$ is small, can we find a topological minor of $K_{D+1}$ in an $(n,d,\lambda)$-graph? In the rest of the section we show that the answer to the latter question is negative; we give constructions of graphs which show that one cannot do much better than what is given in Theorem \ref{thm:minors in (N,D,lambda)}. This is related to the following question of Fountoulakis, K\"uhn and Osthus \cite{FKO09}. For which values $\alpha,k,d$ does every $(\alpha n, k)$-expanding $d$-regular graph on $n$ vertices contain a topological minor of $K_{d+1}$? We will show that even if we assume strong expansion properties, this still does not force a topological minor of $K_{d+1}$; in particular, for $d\geq3$ and infinitely many values of $n$ we show that there is an $(\alpha n, d-3)$-expanding $(2d+1)$-regular graph, where $\alpha=\alpha(d)$, with no subdivision of $K_{2d+2}$. First, we need the following definition.

\begin{Definition}
 The strong product of two graphs $G_1$ and $G_2$ is the graph $G_1\boxtimes G_2$ on the vertex set $V(G_1)\times V(G_2)$, where $(u_1,u_2)$ is adjacent to $(v_1,v_2)$ if and only if one of the following holds:
 \begin{itemize} 
 \item $u_1=v_1$ and $u_2\sim v_2$ in $G_2$;
 \item $u_2=v_2$ and $v_1\sim u_1$ in $G_1$;
 \item $u_1\sim v_1$ in $G_1$ and $v_2\sim u_2$ in $G_2$.
 \end{itemize}
 Given a graph $G$, we denote with $G^{(k)}$ the strong product of $G$ and $K_k$: $G^{(k)}=G\boxtimes K_k$.
\end{Definition}

\begin{lemma}\label{lem:construction}
 Let $G$ be a triangle-free $d$-regular graph for $d\geq7$, and let $k\geq 2$. Then $G^{(k)}$ does not contain a topological minor of $K_t$ for $t=kd+2$.
\end{lemma}

\begin{proof}
Observe that $G^{(k)}$ can be obtained from $G$ by replacing each vertex $v$ in $G$ by a clique $C_v$ of size $k$, and then making all vertices in $C_v$ adjacent to all vertices in $C_u$ if $v\sim u$ in $G$. Hence, $G^{(k)}$ is $(kd+k-1)$-regular.
Now, suppose to the contrary, that there is a topological minor $T$ of $K_t$ in $G^{(k)}$. Let $v$ be one of its branch vertices. We show that there is at least another vertex $w\in C_v$ which is a branch vertex in $T$. If this is not the case, then every vertex $w$ in $C_v$ is either not a neighbor of $v$ in $T$, or it is one of the internal vertices in $T$ adjacent to $v$. In the latter case, this means that the path in $T$, starting at $v$ and continuing with $w$, at some point leaves $C_v$ and lands in one of the $kd$ other neighbors $w'$ of $v$ in $G^{(k)}$. This implies that $vw'$ is not an edge in $T$, since every cycle in a topological minor contains at least $3$ branching vertices. Note also that for distinct vertices $w_1,w_2\in C_v$ in the latter case, $w_1',w_2'$ are also distinct, as otherwise we again get a cycle with less than $3$ branching vertices.

In both cases discussed in the passage above, for every vertex in $w\in C_v-v$, one of the $kd+k-1$ neighbors of $v$ in $G^{(k)}$ is not its neighbor in $T$, and thus $v$ has at most $kd+k-1-(|C_v|-1)=kd=t-2$ neighbors in $T$, a contradiction. Hence, we may assume that there is another branching vertex $w\in C_v$ and without loss of generality, let $vw$ be an edge in $T$. Consider the (common) neighborhood in $G^{(k)}-C_v$ of $v$ and $w$, and note it is of size $kd$ and that it is a disjoint union of cliques of size $k$, since by assumption $G$ does not contain triangles. Since both $v$ and $w$ have at least $(t-1)-(k-1)> kd-k$ neighbors in $T-C_v$ each, it means they have at least $kd-2k>d$ common neighbors in $T-C_v$, denoted by $N$. Note that all vertices in $N$ are also branching vertices. Now, by the pigeonhole principle, there exist two vertices $x,x'$ in $N$ from the same $k$-clique in $G^{(k)}$. In order to get to the other at least $|N|-|C_x|> kd-3k$ branching vertices in $N\setminus C_x$, $x$ and $x'$ must use internally vertex disjoint paths of length at least two (indeed, recall that no vertex in $C_x$ is adjacent to any vertex in $N\setminus C_x$ in the graph $G^{(k)}$). This is in turn not possible, since $x$ and $x'$ have an identical neighborhood (up to $x,x'$), which is of size $kd+k-1$, so one of them will use less than half of those to get to each vertex in $N\setminus C_x$, which is a contradiction since $|N\setminus C_x|>kd-3k>\frac{kd+k-1}{2}$, for $d\geq7$. This completes the proof. 
\end{proof}

The next lemma is a standard exercise in spectral graph theory (see for example Problem 11.7 in \cite{LO}).
\begin{lemma}\label{lem:eigenvalues}
 Let $\{\lambda_i\}_{i\in I}$ be the eigenvalues of graph $G_1$, and $\{\mu_j\}_{j\in J}$ be the eigenvalues of graph $G_2$. Then the eigenvalues of $G_1\boxtimes G_2$ are $\lambda_i \mu_j+\lambda_i+\mu_j$, for $(i,j)\in I\times J$.
\end{lemma}
Using Lemmas \ref{lem:construction} and \ref{lem:eigenvalues}, we get the following result. 

\begin{thm}
Let $G$ be a triangle-free $(n,d,\lambda)$-graph with $7\leq d\leq n-1$ and let $k\geq 2$. Then $G^{(k)}$ is a $(kn,kd+k-1,k\lambda+k-1)$-graph without a topological minor of $K_t$ for $t=kd+2$.
\end{thm}
\begin{proof}
By Lemma \ref{lem:construction} the graph $G^{(k)}$ does not contain a topological minor of $K_t$. It is left to prove that the second largest eigenvalue in absolute value of $G^{(k)}$ is at most $k\lambda+k-1$. First, it is an easy exercise to show that the eigenvalues of $K_k$ are $k-1$ with multiplicity 1, and $-1$ with multiplicity $k-1$. On the other hand, the largest eigenvalue of $G$ is $d$ while the other eigenvalues are by assumption at most $\lambda$ in absolute value. 
By applying Lemma \ref{lem:eigenvalues}, one can see that all except the largest eigenvalue of $G^{(k)}$ are at most $\lambda (k-1)+\lambda+(k-1)$ in absolute value, so we are done.
\end{proof}

\begin{remark}
Recall that in the introduction we mentioned that there exist triangle-free $(n,d,\lambda)$-graphs with $\lambda=O(\sqrt{d})$. By assuming that $\lambda=O(\sqrt{d})$ in the previous theorem, one gets a $(n_1,d_1,\lambda_1)$-graph where $n_1=kn$, $d_1=kd+k-1$ and $\lambda_1=k\lambda +k-1$, with no topological minor of $K_t$ with $t=d_1-(k-3)=d_1-\Theta\left(\frac{\lambda_1^2}{d_1}\right)$. Recall that by Theorem \ref{thm:minors in (N,D,lambda)} one is guaranteed to find a topological minor of $K_t$ for $t=d_1-\Theta(\lambda_1)$ whenever $\lambda_1<d_1/240$; by fixing $d$ and $\lambda$, and choosing $k$ to be large in the previous theorem, one gets $t=d_1-\Theta(\lambda_1)$, matching the mentioned lower bound for $t$. 
\end{remark}

Finally, we have the following corollary of Lemma \ref{lem:construction}, which gives a concrete bound for the parameters in the abovementioned question of Fountoulakis, K\"uhn and Osthus, putting an emphasis on the expansion ratio, which is close to half of the degree of the constructed regular graph.
\begin{corollary}
 For every $d\geq 4$ there exists a constant $\alpha>0$, such that for infinitely many values of $n$ there is a $(\alpha n, d-3)$-expanding $(2d+1)$-regular graph on $n$ vertices without a topological minor of $K_{2d+2}$.
\end{corollary}
\begin{proof}
Let $G$ be a random bipartite $d$-regular graph on an even number $n/2$ of vertices; such a graph is whp $(\alpha n, d-3)$-expanding for constant $\alpha=\alpha(d)$ (see, for example, Theorem 4.16 in \cite{hoory2006expander}). By Lemma \ref{lem:construction}, the graph $G^{(2)}$  does not contain a $K_{d+1}$-topological minor; furthermore, one can easily see that $G^{(2)}$ is also $(\alpha n, d-3)$-expanding, so we are done.
\end{proof}

\bibliographystyle{plain}

\section{Appendix}
\begin{proof}[Proof of Theorem \ref{thm:FP}]
Let $F'$ be a graph obtained from $F$ by adding a leaf $v$ to a vertex $w\in V(F)$ such that $d_F(w)\leq D-1$ and $v(F')\leq s$, and let us show that an embedding $\phi'$ as described exists; the theorem then follows by induction.
For a set $X\subset V(G)$ and an embedding $f:H\hookrightarrow G$ let 
$$R(X,f)=	|\Gamma_G(X) \setminus f(H)| -\sum_{v \in X} \left[ D - \deg_H(f^{-1}(v)) \right]-|f(H)\cap X|.$$
Let $Y=\Gamma_G(\phi(w))\setminus \phi(F)$; for every $a\in Y$ we can extend $\phi$ to an embedding $\phi_a$ of $F'$ by letting $\phi_a(v)=a$. 
We need to show that there is an $a\in Y$ such that $R(X,\phi_a)\geq0$ for all $X$ with $|X|\leq 2s-2$.
Suppose to the contrary that for every $a\in Y$ there exists a set $X_a$ with $|X_a|\leq 2s-2$ such that $R(X_a,\phi_a)<0$. By assumption $\phi$ is $(2s-2,D)$-good, so it holds that $R(X,\phi)\geq0$ for all sets $X$ with $|X|\leq 2s-2$. 
Note that for every $a\in Y$ we have:
\begin{align*}
R(X_a,\phi_a)-R(X_a,\phi)&=-\mathds{1}\big[a\in \Gamma_G(X_a)\big]+\mathds{1}\big[a\in X_a\big]+\mathds{1}\big[\phi(w)\in X_a\big]-\mathds{1}\big[a\in X_a\big]\\
&=\mathds{1}\big[\phi(w)\in X_a\big]-\mathds{1}\big[a\in \Gamma_G(X_a)\big]
\end{align*}


Hence, in order to have $R(X_a,\phi_a)<0$, we may assume that $R(X_a,\phi)=0$, $\phi(w)\notin X_a$ and $a\in \Gamma_G(X_a)$.
We will also need the following claims.
\begin{claim}\label{claim1}
If $X\subseteq V(G)$ is such that $R(X,\phi)=0$ and $|X|\leq 2s-2$ then $|X|\leq s-1$.
\end{claim}
\begin{proof}
Since $G$ is $(2s-2,D+2)$-expanding, we have that $|\Gamma_G(X)|\geq (D+2)|X|$; we also know that $v(F)\leq s-1$. Thus we have
\begin{equation*}
    0=R(X,\phi) \geq (D+2)|X|-(s-1)-|X|D-|X|=|X|-(s-1),
\end{equation*}
so the claim follows.
\end{proof}
\begin{claim}\label{claim2}
The function $R(\,\cdot\,,\phi)$ is submodular, i.e., for all sets of vertices $A,B$ it holds that $R(A\cup B,\phi)+R(A\cap B,\phi)\leq R(A,\phi)+R(B,\phi)$.
\end{claim}
\begin{proof}
The first term of $R(X,\phi)$ is a submodular function of $X$, and the other two are modular, so the claim follows.
\end{proof}
\begin{claim}\label{claim3}
If for $A,B\subseteq V(G)$ it holds that $R(A,\phi)=R(B,\phi)=0$, and $|A|,|B|\leq s-1$, then $R(A\cup B,\phi)=0$ and $|A\cup B|\leq s-1$.
\end{claim}
\begin{proof}
Since $\phi$ is $(2s-2,D+2)$-good and $|A\cup B|, |A\cap B|\leq 2s-2$, we have that $R(A\cup B,\phi),R(A\cap B,\phi)\geq 0$. Now it follows by Claim \ref{claim2} that 
$R(A\cup B,\phi)=0$, and in turn, by Claim \ref{claim1} we have $|A\cup B|\leq s-1$.
\end{proof}
Now we are ready to finish the proof. Recall that for every $a\in Y$ it holds that $|X_a|\leq 2s-2$ and $R(X_a,\phi)=0$. Let $X^*=\bigcup_{a\in Y} X_a$ and let us show that $R(X^*,\phi)=0$, and thus by Claim \ref{claim1} we will have $|X^*|\leq s-1$. If $Y=\emptyset$, we are done, and otherwise we get by induction and Claim \ref{claim3} that $R(X^*,\phi)=0$. Now we consider the set $X'=X^*\cup \{\phi(w)\}$, and notice that $X^*\subsetneq X'$ since we showed that $\phi(w)\notin X_a$ for every $a\in Y$, so $\phi(w)\notin X^*$. Since $a\in \Gamma_G(X_a)$ for every $a\in Y$, this means that $\Gamma_G(\phi(w))\setminus \phi(F)=Y\subseteq \Gamma_G(X^*)$, hence the first term in $R(X^*,\phi)$ and the first term in $R(X',\phi)$ are the same. Furthermore, because of the second term, we have $R(X',\phi)\leq R(X^*,\phi)-\big(D-\deg_F(w)\big)\leq R(X^*,\phi)-1<0$ as $\phi(w)\in X'$ and $\phi(w)\notin X^*$, a contradiction with $R(X',\phi)\geq 0$.
\end{proof}

\end{document}